\documentclass[12pt]{amsart}
\usepackage{a4wide}
\usepackage{amssymb}
\usepackage{amsthm}
\usepackage{amsmath}
\usepackage{amscd}
\usepackage{verbatim}
\usepackage{color}
\usepackage[mathscr]{eucal}

\newtheorem{theorem}{Theorem}
\newtheorem{lemma}[theorem]{Lemma}

\newtheorem{proposition}[theorem]{Proposition}

\theoremstyle{remark}

\newtheorem*{remark}{Remark}
\newtheorem*{remarks}{Remarks}

\numberwithin{theorem}{section} \numberwithin{equation}{section}

\newcommand{\R}{\mathbb{R}}
\newcommand{\C}{\mathbb{C}}

\newcommand{\Z}{\mathbb{Z}}
\newcommand{\N}{\mathbb{N}}
\newcommand{\SL}{{\text {\rm SL}}}

\renewcommand{\H}{\mathbb{H}}
\newcommand{\sgn}{\operatorname{sgn}}
\newcommand{\erf}{\operatorname{erf}}

\newcommand{\fkD}{f_{D}}

\newcommand{\QD}{\mathcal{Q}_D}

\newcommand{\re}{\text{Re}}
\newcommand{\im}{\text{Im}}
\newcommand{\Th}{\Theta_1}
\newcommand{\Thstr}{\Theta_2}
\newcommand{\spc}{\mathfrak{M}}
\begin{document}

\title[A generating function of locally harmonic Maass forms]{On a completed generating function of locally harmonic Maass forms}
\author{Kathrin Bringmann}
\email{kbringma@math.uni-koeln.de}
\address{Mathematical Institute\\University of Cologne\\ Weyertal 86-90 \\ 50931 Cologne \\Germany}
\author{Ben Kane}
\email{bkane@math.uni-koeln.de}
\address{Mathematical Institute\\University of Cologne\\ Weyertal 86-90 \\ 50931 Cologne \\Germany}
\author{Sander Zwegers}
\email{szwegers@math.uni-koeln.de}
\address{Mathematical Institute\\University of Cologne\\ Weyertal 86-90 \\ 50931 Cologne \\Germany}
\thanks{The research of the first author was supported by the Alfried Krupp Prize for Young University Teachers of the Krupp Foundation.}
\begin{abstract}
While investigating the Doi--Naganuma lift, Zagier defined integral weight cusp forms $f_D$ which are naturally defined in terms of binary quadratic forms of discriminant $D$.  It was later determined by Kohnen and Zagier that the generating function for the $f_D$ is a half-integral weight cusp form.  A natural preimage of $f_D$ under a differential operator at the heart of the theory of harmonic weak Maass forms was determined by the first two authors and Kohnen.  In this paper, we consider the modularity properties of the generating function of these preimages.  We prove that although the generating function is not itself modular, it can be naturally completed to obtain a half-integral weight modular object.
\end{abstract}
\date{\today}
\maketitle

\section{Introduction and statement of results}
\noindent
Throughout we let $k\geq 4$ be an even integer.  While investigating the Doi--Naganuma lift \cite{DoiNaganuma}, Zagier \cite{ZagierDN} defined for $D>0$ the function
$$
f_{D}(\tau):= f_{k,D}( \tau):= D^{k-\frac{1}{2}}\sum_{Q\in \QD}\frac{1}{Q\left(\tau,1\right)^k} \qquad (\tau\in \H),
$$
where $\QD$ is the set of integral binary quadratic forms
$$
[a,b,c](X,Y):=aX^2+bXY+cY^2
$$
of discriminant $D\in\Z$.  Note that $f_D$ has been renormalized from Zagier's original definition.  The function $f_{D}$ is a weight $2k$ cusp form, while the generating function 
$$
\Omega(\tau,z):=\sum_{D>0} f_{D}(\tau) e^{2\pi i Dz} \qquad (z\in\H)
$$
for the $f_D$ is a modular form of weight $k+\frac{1}{2}$ in the $z$ variable \cite{KohnenZagier}.  As was shown by the first two authors and Kohnen in \cite{BKW}, the functions $f_D$ have natural weight $2-2k$ preimages $\mathcal{F}_D=\mathcal{F}_{1-k,D}$ under the operator $\xi_{2-2k}:=2iy^{2-2k}\overline{\frac{\partial}{\partial \overline{\tau}}}$, which is central in the theory of harmonic weak Maass forms.  In this paper we investigate the modularity properties of the generating function
$$
\Psi\left(\tau,z\right):=\sum\limits_{D>0}\mathcal{F}_{D}\left(\tau\right)e^{2\pi i Dz}.
$$
In \cite{BKW}, it was shown that $\mathcal{F}_D$ exhibits discontinuities along the geodesics defined by 
$$
E_D:=\left\{ \tau=x+iy\in \H: \exists a,b,c\in \Z,\ b^2-4ac=D,\ a\left|\tau\right|^2+bx+c=0\right\}.
$$
Hence, as a function of $\tau$, the set of discontinuities of $\Psi$ is dense in the upper half plane.  Nonetheless, the function is still well defined at each point $\tau\in\H$ and we investigate the modularity property of $\Psi$ as a function of $z$ whenever $\tau$ is fixed.  However, unlike in the case of $\Omega$, $\Psi$ is not itself modular, but may be naturally completed to a function which is modular of weight $\frac{3}{2}-k$ as a function of $z$.  This mirrors the mock theta functions of Ramanujan, which are themselves holomorphic but may be completed to nonholomorphic modular objects called harmonic weak Maass forms.  The mock theta functions are in a class of functions called \begin{it}mock modular forms\end{it}, which have naturally appeared in a variety of applications.  Their benefit has been observed in the areas of partition theory (for example \cite{AndrewsMock,Bringmann,BGM,BOfq,BORank}), Zagier's duality \cite{ZaDual} (for example \cite{BODual}), and derivatives of $L$-functions (for example \cite{BrO,BrY}).  To give another example, they have also recently appeared in Eguchi, Ooguri, and Tachikawa's \cite{EOT} investigation of moonshine
 for the largest Mathieu group $M_{24}$.  For a good overview of mock modular forms, see \cite{OnoUnearth} and \cite{ZagierBourbaki}.

We now return to the properties of the functions $\mathcal{F}_{D}$.  In addition to being natural preimages of the functions $\fkD$, the $\mathcal{F}_D$ are furthermore \begin{it}locally harmonic Maass forms.\end{it}  Such functions satisfy weight $2-2k$ modularity and are annihilated (away from a certain set of measure zero) by the weight $2-2k$ hyperbolic Laplacian
$$
\Delta_{2-2k}:= -y^2\left( \frac{\partial^2}{\partial x^2}+\frac{\partial^2}{\partial y^2}\right) + i\left(2-2k\right) y\left(\frac{\partial}{\partial x}+i \frac{\partial}{\partial y}\right)\qquad (\tau=x+iy).
$$
Denoting for $Q=[a,b,c]\in \QD$
$$
Q_{\tau}:=\frac{1}{y}\left(a|\tau|^2+bx+c\right),
$$
the functions $\mathcal{F}_D$ are explicitly defined by
$$
\mathcal{F}_{D}\left(\tau\right):=\frac{2}{\beta\left(k-\frac{1}{2},\frac{1}{2}\right)}\sum_{Q\in \QD}\sgn\left(Q_{\tau}\right) Q\left(\tau,1\right)^{k-1} \psi_k\left(\frac{Dy^2}{\left|Q\left(\tau,1\right)\right|^2_{\phantom{-}}}\right).
$$
Here we use the convention $\sgn(0)=0$ and 
$$
\psi_k\left(v\right):=\frac{1}{2}\beta\left(v;k-\frac{1}{2},\frac{1}{2}\right)
$$
is a  special value of the incomplete $\beta$-function, which is given for $s,w\in \C$ satisfying $\re\left(s\right)$, $\re\left(w\right)>0$ by
$$
\beta\left(v;s,w\right):=\int_{0}^v t^{s-1}\left(1-t\right)^{w-1}dt.
$$
  Moreover, for $\re(s),\re(w)>0$, we denote $\beta(s,w):=\beta(1;s,w)$.  Note that we have renormalized the definition of $\mathcal{F}_D$ given in \cite{BKW}.

To complete $\Psi$, we define for $D\in \Z$
\begin{equation}\label{eqn:GDdef}
\mathcal{G}_{D}\left(v;\tau\right):=-\frac{1}{\sqrt{\pi}}\sum\limits_{Q\in \QD}\sgn\left(Q_{\tau}\right)Q\left(\tau,1\right)^{k-1} \Gamma\left(\frac{1}{2};4\pi Q_{\tau}^2 v\right) \qquad (z=u+iv),
\end{equation}
where
$$
\Gamma(s;w):=\int_{w}^{\infty} t^{s-1}e^{-t}dt\qquad (w>0,\ s\in \C)
$$
is the incomplete gamma function.   We denote the generating function for the $\mathcal{G}_D$ by
$$
\Psi^*\left(\tau,z\right):=\sum_{D\in \Z} \mathcal{G}_D\left(v;\tau\right)e^{2\pi i Dz}.
$$
We then define the completion of $\Psi$ by
\begin{equation}\label{eqn:Thetakdef}
\widehat{\Psi}(\tau,z):=\Psi(\tau,z)+\Psi^*\left(\tau,z\right).
\end{equation}
Note that, as a function of $\tau$, the function $-\mathcal{G}_D$ exhibits the same singularities as $\mathcal{F}_D$, and hence the singularities vanish when summing them together.  As a result, the function $\widehat{\Psi}$ is real analytic in both variables.  To state the modularity properties of $\widehat{\Psi}$, we set for $\kappa\in \frac{1}{2}\Z$
$$
\Gamma:=\begin{cases}
\SL_2(\Z)&\text{if }\kappa\in \Z,\\
\Gamma_0(4)&\text{if }\kappa\in \frac{1}{2}\Z\setminus\Z.
\end{cases}
$$
Let $\spc_{\kappa}$ denote the space of real analytic functions $f:\H\to \C$ satisfying weight $\kappa\in \frac{1}{2}\Z$ modularity for $\Gamma,$ with the additional restriction that $f$ is in Kohnen's plus space whenever $\kappa\in \frac{1}{2}\Z\setminus\Z$.  For a formal definition of Kohnen's plus space, see the comments preceding Lemma \ref{lem:lowermodular}.
\begin{theorem}\label{thm:Thetamodular}
As a function of $z$, $\widehat{\Psi}(\tau,z)$ is an element of $\spc_{k+\frac{1}{2}}$, while as a function of $\tau$ it is an element of $\spc_{2-2k}$.
\end{theorem}

\begin{remarks}

\noindent
\begin{enumerate}
\item
 
One can show that, as a function of $z$, $\widehat{\Psi}$ satisfies the growth conditions of a cusp form, i.e., $v^{\frac{k}{2}+\frac{1}{4}}|\widehat{\Psi}(\tau,z)|$ is bounded on $\H$.  However, the corresponding growth condition in $\tau$ is not satisfied by $\widehat{\Psi}$.
\item
Since the functions $\mathcal{F}_D$ exhibit discontinuities along certain geodesics, it is somewhat surprising that $\widehat{\Psi}$ is real analytic in $\tau$.
\end{enumerate}
\end{remarks}
The function $\widehat{\Psi}$ is furthermore naturally connected to $\Omega$ and indefinite theta functions through the weight lowering operator $L_{w}:=\im(w)^2\frac{\partial}{\partial\overline{w}},$ which sends functions satisfying weight $\kappa$ modularity to functions which satisfy weight $\kappa-2$ modularity.  The theta functions we require are
$$
\Th\left(\tau,z\right):=iv^{\frac{3}{2}}\sum_{\substack{D\in \Z\\ Q\in \QD}} Q\left(\tau,1\right)^{k-1} Q_{\tau} e^{-4\pi Q_{\tau}^2 v} e^{2\pi i Dz}
$$
and (the projection into Kohnen's plus space of) Shintani's \cite{Shintani} (non-holomorphic) classical theta kernel
$$
\Thstr\left(\tau,z\right):=2iv^{\frac{1}{2}}y^{-2k}\sum_{\substack{D\in \Z\\ Q\in \QD}}Q(\tau,1)^k e^{-4\pi Q_{\tau}^2 v}e^{2\pi i Dz}.
$$
Although these two theta functions have known modularity properties, we supply direct proofs of this modularity in Lemma \ref{lem:lowermodular} as a convenience to the reader.  Specifically, the function $\Th\left(\tau,z\right)$ is a weight $k-\frac{3}{2}$ indefinite theta function for $\Gamma_0(4)$ in Kohnen's plus space in $z$ and satisfies weight $2-2k$ modularity for $\SL_2(\Z)$ in $\tau$.  The function $\Thstr\left(-\overline{\tau},z\right)$ is a weight $k+\frac{1}{2}$ indefinite theta function for $\Gamma_0(4)$ in Kohnen's plus space in $z$ and satisfies weight $2k$ modularity for $\SL_2(\Z)$ in $\tau$.
\begin{theorem}\label{thm:Thetaops}
\noindent

\noindent
\begin{enumerate}
\item
The image of the function $\widehat{\Psi}$ under the lowering operator in $z$ equals
\begin{equation}\label{eqn:lowerz}
L_z\left( \widehat{\Psi}\left(\tau,z\right)\right) = \Th\left(\tau,z\right).
\end{equation}
\item
The image of the function $\widehat{\Psi}$ under the $\xi$-operator in $\tau$ equals
\begin{equation}\label{eqn:lowertau1}
\xi_{2-2k,\tau}\left(\widehat{\Psi}(\tau,z)\right)= 2i\left(\Theta_2(-\overline{\tau},-\overline{z}) -\frac{i}{\beta\left(k-\frac{1}{2},\frac{1}{2}\right)}\Omega_k\left(\tau,-\overline{z}\right)\right)
\end{equation}
\end{enumerate}
\end{theorem}

\begin{remark}
Bruinier, Funke, and Imamo$\overline{\text{g}}$lu communicated to us that they obtained analogous results to our Theorems \ref{thm:Thetamodular} and \ref{thm:Thetaops} for the case $k=0$ \cite{BIF2}.  Their approach is based on extending the theta lift considered in \cite{BIF1} to meromorphic modular functions.
\end{remark}
The paper is organized as follows.  In Section \ref{sec:setup} we use a theorem of Vign\'eras \cite{Vigneras} to supply a direct proof of the modularity of $\Th$ and $\Thstr$.  Section \ref{sec:holproj} is devoted to holomorphic projection, which is a key ingredient in the proof of Theorem \ref{thm:Thetamodular}.  Section \ref{sec:converge} is centered around the convergence of $\widehat{\Psi}$ and its real analyticity.  The modularity of $\widehat{\Psi}$ is established in Section \ref{sec:modularity}.

\section*{Acknowledgements}
The authors thank the referee for many helpful comments.

\section{Indefinite theta functions}\label{sec:setup}
For $\kappa\in\frac{1}{2}\Z$, a finite index subgroup $\Gamma\subseteq \SL_2(\Z)$, and a character $\chi$, we say that a function $f:\H\to \C$ is
\begin{it}modular of weight $\kappa$ for $\Gamma$ with character $\chi$\end{it} if for every $M=\left(\begin{smallmatrix}\alpha&\beta\\ \gamma&\delta\end{smallmatrix}\right)\in \Gamma$ one has $f|_{\kappa} M =\chi\left(\delta\right) f$.  Here $|_{\kappa}$ is the usual weight $\kappa$ slash operator.

To show the modularity of the indefinite theta functions which we encounter in this paper, we will employ a result of Vign\'eras \cite{Vigneras}.  For this, we define the Euler operator $E:=\sum_{i=1}^{n} w_i\frac{\partial}{\partial w_i}$.  As usual, we denote the Gram matrix associated to a nondegenerate quadratic form $q$ on $\R^n$ by $A$.  The \begin{it}Laplacian\end{it} associated to $q$ is then defined by $\Delta:=\left\langle \frac{\partial}{\partial w}, A^{-1}\frac{\partial}{\partial w}\right\rangle$.  Here $\left<\cdot,\cdot\right>$ denotes the usual inner product on $\R^n$.

\begin{theorem}[Vign\'eras]\label{thm:Vigneras}
Suppose that $n\in \N$, $q$ is a nondegenerate quadratic form on $\R^n$, $L\subset \R^n$ is a lattice on which $q$ takes integer values, and $p:\R^n\to\C$ is a function satisfying the following conditions:
\begin{itemize}
\item[(i)]
The function $f(w):=p(w)e^{-2\pi q(w)}$ times any polynomial of degree at most 2 and all partial derivatives of $f$ of order at most $2$ are elements of $L^2\left(\R^n\right)\cap L^1\left(\R^n\right)$.
\item[(ii)]
For some $\lambda\in\Z$, the function $p$ satisfies
\[
\left(E-\frac{\Delta }{4\pi}\right)p=\lambda p. 
\]
\end{itemize}
Then the indefinite theta function
\[
v^{-\frac{\lambda}{2}}\sum_{w\in L} p\left(w\sqrt{v}\right) e^{2\pi i q(w)z}
\]
is modular of  weight $\lambda+\frac{n}{2}$ for $\Gamma_0(N)$ and character $\chi\cdot \chi_{-4}^{\lambda}$, where $N$ and $\chi$ are the level and character of $q$ and $\chi_{-4}$ is the unique primitive Dirichlet character of conductor $4$.
\end{theorem}

We use Theorem \ref{thm:Vigneras} to show the modularity of the theta functions $\Th$ and $\Thstr$.  To state the modularity, recall that $\spc_{\kappa}$ denotes the space of real analytic functions $f:\H\to \C$ satisfying weight $\kappa\in \frac{1}{2}\Z$ modularity for $\Gamma,$ with the additional restriction that $f$ is in Kohnen's plus space whenever $\kappa\in \frac{1}{2}\Z\setminus\Z$.  
Here we say that a function satisfying weight $\ell+\frac{1}{2}$ modularity is an element of Kohnen's plus space if its Fourier expansion has the shape
$$
\sum_{(-1)^{\ell} n\equiv 0,1\pmod{4}} a_n(v)e^{2\pi i nz}.
$$
\begin{lemma}\label{lem:lowermodular}
\noindent

\noindent
\begin{enumerate}
\item
As a function of $z$, $\Th\left(\tau,z\right)\in \spc_{k-\frac{3}{2}}$.   Furthermore, as a function of $\tau$, $\Th\left(\tau,z\right)\in \spc_{2-2k}$.
\item
As a function of $z$, $\Thstr\left(-\overline{\tau},z\right)\in \spc_{k+\frac{1}{2}}$.  As a function of $\tau$, $\Thstr\left(-\overline{\tau},z\right)\in \spc_{2k}$.
\end{enumerate}
\end{lemma}
\begin{proof}

\noindent
Since the proofs are entirely analogous, we only show part (1).  To prove the modularity in $z$, we use Theorem \ref{thm:Vigneras} with $q(a,b,c)=b^2-4ac$, $L=\Z^3$, and
\[
p(a,b,c)=Q_{\tau}Q(\tau,1)^{k-1}e^{- 4\pi Q_{\tau}^2}.
\]
One sees directly that 
$$
p\left(\sqrt{v}a,\sqrt{v}b,\sqrt{v}c\right)= v^{\frac{k}{2}} Q_{\tau}Q(\tau,1)^{k-1}e^{- 4\pi Q_{\tau}^2v}.
$$
We next note that
\begin{equation}\label{eqn:Qrewrite}
\left|Q(\tau,1)\right|^2= Q_{\tau}^2 y^2 +Dy^2.
\end{equation}
It is then straightforward to show that
\begin{equation}\label{eqn:Qposdef}
D +2Q_{\tau}^2 =\frac{2}{y^2}\left|Q(\tau,1)\right|^2-D
\end{equation}
is positive definite.
From this, one can easily verify that condition (i) of Theorem \ref{thm:Vigneras} is satisfied.

A straightforward calculation yields
\[
E\left(p(a,b,c)\right)=\left(k-8\pi Q_{\tau}^2 \right)p(a,b,c)
\]
and
\[
\Delta\left(p(a,b,c)\right)=4\pi\left(3-8\pi Q_{\tau}^2 \right)p(a,b,c).
\]
Thus $\lambda=k-3$ in Theorem \ref{thm:Vigneras}.  This gives that $\Th\left(\tau,z\right)$ is an indefinite theta function of weight $k-\frac{3}{2}$ for $\Gamma_0(4)$.  Since $k-2$ is even, one sees that the plus space condition is clearly satisfied for $\Th$.

To prove the modularity in the $\tau$ variable, we directly apply translation and inversion.  By making the change of variables $b\to b+2a$ and $c\to a+b+c$, one sees by term by term comparison that
$$
\Th\left(\tau+1,z\right)=\Th\left(\tau,z\right).
$$
Similarly, the change of variables $a\to c$, $b\to -b$, and $c\to a$ implies that
$$
\Th\left(-\frac{1}{\tau},z\right) = \tau^{2-2k}\Th\left(\tau,z\right).
$$
\end{proof}

\section{Holomorphic projection}\label{sec:holproj}
In this section we introduce the holomorphic projection operator and investigate some of its basic properties.  In the integer weight case, these properties were first proven by Sturm \cite{Sturm} and a good overview may be found in Appendix C of \cite{ZagierIntro}.  For a translation invariant function $f:\H\to\C$ we write its Fourier expansion as
\begin{equation}\label{eqn:Fourierv}
f(z)=\sum_{r\in\Z} c_{r}(v) e^{2\pi i r z}.
\end{equation}
We formally define the \begin{it}weight $\kappa$ holomorphic projection of $f$\end{it} by
$$
\pi_{\kappa}(f)(z):=\pi_{\kappa,z}(f)(z) :=\sum_{r\in \N} c_{r} e^{2\pi i r z},
$$
where
\begin{equation}\label{eqn:holprojformal}
c_{r}:=\frac{(4\pi r)^{\kappa-1}}{\Gamma\left(\kappa -1\right)}\int_0^\infty c_r(t) e^{-4\pi r t} t^{\kappa-2}dt.
\end{equation}
Here $\Gamma(s)$ is the usual gamma function.

It is useful to have the following reformulation of the holomorphic projection operator.
\begin{lemma}\label{Bkernel}
If $f:\H\to \C$ is a translation invariant function, then
\begin{equation}\label{eqn:holproj}
\pi_{\kappa}(f)(z)=\frac{(\kappa-1)(2i)^\kappa}{4\pi}\int_{\mathbb{H}}\frac{f(\tau)y^\kappa}{(z-\overline{\tau})^\kappa}\frac{dxdy}{y^2},
\end{equation}
for every $1<\kappa\in \frac{1}{2}\Z$ for which the right hand side of \eqref{eqn:holproj} converges absolutely.
\end{lemma}
\begin{remark}
In the case that $\kappa\in \Z$, the integral in Lemma \ref{Bkernel} appears in the proof of the trace formula for the Hecke operators  established in \cite{ZagierLang}.
\end{remark}
\begin{proof}
Using the fact that $f$ is translation invariant, we rewrite the integral on the right hand side of \eqref{eqn:holproj} as
$$
\int_{0}^{\infty}y^{\kappa-2}\int_{0}^{1} f(x+iy)  \sum_{n\in \Z}\frac{1}{\left(z-x+iy +n\right)^{\kappa} }  dx dy.
$$
After inserting the Fourier expansion of $f$, the result follows by a special case of the Lipschitz summation formula \cite{Lipschitz}, which yields
$$
\sum_{n\in \Z}\frac{1}{\left(w +n\right)^{\kappa} }= \frac{\left(-2\pi i\right)^{\kappa}}{\Gamma\left(\kappa\right)} \sum_{n\in \N} n^{\kappa-1} e^{2\pi i nw}\qquad\qquad (w\in \H).
$$
\end{proof}
We henceforth extend the definition of the holomorphic projection operator to be the right hand side of \eqref{eqn:holproj} for every (not necessarily translation invariant) function $f$ for which the integral converges absolutely.  Note that the image of any such function is clearly holomorphic by either definition of the holomorphic projection operator.  Indeed, the holomorphic projection operator acts trivially on holomorphic functions.
\begin{lemma}\label{lem:hol}
If $f:\H\to\C$ is holomorphic and the right hand side of \eqref{eqn:holproj} converges absolutely, then one has
$$
\pi_{\kappa}(f)=f.
$$
\end{lemma}
\begin{proof}
We follow the proof of Proposition 1 in Section 6 of \cite{Klingen}.  We make the change of variables $\zeta =\frac{\tau-z}{\tau-\overline{z}}$ in \eqref{eqn:holproj} to get
\begin{equation*}
\begin{split}
\pi_\kappa(f)(z) &= \frac{\kappa -1}{4\pi} \int_\H f(\tau)\left(\frac{2iy}{z-\overline{\tau}}\right)^\kappa \frac{dx dy}{y^2}\\
&= \frac{\kappa -1}{\pi} \int_{B_1} \widetilde{f} (\zeta) \left(1-|\zeta|^2\right)^{\kappa -2} d \zeta_1 d \zeta_2,
\end{split}
\end{equation*}
where $\zeta=\zeta_1+i\zeta_2$, $B_1 := \left\{ \zeta\in\C \big| |\zeta|<1\right\}$ and
$$ 
\widetilde{f}(\zeta) = (1-\zeta)^{-\kappa} f\left(\frac{z-\zeta \overline{z}}{1-\zeta}\right).
$$
The function $\widetilde{f}$ is holomorphic on $B_1$ since $f$ is holomorphic on $\H$. Using polar coordinates $\zeta=R e^{i\vartheta}$ we get
\begin{equation*}
\begin{split}
\pi_\kappa(f)(z) &= \frac{\kappa -1}{\pi} \int_{B_1} \widetilde{f} (\zeta) \left(1-|\zeta|^2\right)^{\kappa -2} d\zeta_1 d\zeta_2\\
&= \frac{\kappa -1}{\pi} \int_0^1 R\left(1-R^2\right)^{\kappa -2} \left( \int_0^{2\pi} \widetilde{f}\left(Re^{i\vartheta}\right)d \vartheta\right) dR.
\end{split}
\end{equation*}
Using Cauchy's integral formula we immediately see that for $0\leq R <1$
$$
\frac{1}{2\pi} \int_0^{2\pi} \widetilde{f}\left(Re^{i\vartheta}\right)d\vartheta = \widetilde{f} (0) = f(z),
$$
and so 
$$
 \pi_\kappa(f)(z) = f(z) (\kappa -1) \int_0^1 2R \left(1-R^2\right)^{\kappa -2} d R = f(z).
$$
\end{proof}

An easy change of variables in \eqref{eqn:holproj} immediately implies that holomorphic projection commutes with the weight $\kappa$ slash operator.
\begin{lemma}\label{lem:slash}
If the right hand side of \eqref{eqn:holproj} converges absolutely for $\kappa\in\frac{1}{2}\Z$, then one has for every $M\in \SL_2(\Z)$
$$
\pi_{\kappa}(f)\big|_{\kappa}M= \pi_{\kappa}\left(f\big|_{\kappa}M\right).
$$
\end{lemma}

Combining Lemmas \ref{Bkernel} and \ref{lem:slash} yields the following special case.
\begin{lemma}\label{lem:holconverge}
If $\left|f(z)\right| v^{r}$ is bounded on $\H$ and $\kappa\in\frac{1}{2}\Z$ satisfies $\kappa>  r+1>1$, then for every $M\in \SL_2(\Z)$ one has
\begin{equation}\label{eqn:holslash}
\pi_{\kappa}(f)\big|_{\kappa}M = \pi_{\kappa}\left(f\big|_{\kappa}M\right).
\end{equation}
  Moreover, $\left|\pi_{\kappa}(f)(z)\right|v^r$ is bounded on $\H$.
\end{lemma}
\begin{proof}
Making the change of variables $x\to x(y+v)+u$ and then $y\to yv$, we may bound the integral of the absolute value by
\begin{multline}\label{eqn:intabs}
\int_{0}^{\infty}\frac{y^{\kappa-2}}{(1+y)^{\kappa-1}}\int_{-\infty}^{\infty} \frac{\left|f\left(xv(1+y)+u+ivy\right)\right|}{\left(x^2+1\right)^{\frac{\kappa}{2}}} dxdy\\
\ll v^{-r}\int_{0}^{\infty}\frac{y^{\kappa-r-2}}{(1+y)^{\kappa-1}} dy\int_{0}^{\infty} \frac{1}{\left(x^2+1\right)^{\frac{\kappa}{2}} } dx.
\end{multline}
Here we have used the assumed bound for $f$.  The integral over $x$ converges for $\kappa>1$ and the integral over $y$ converges for $\kappa>r+1>1$.  Lemma \ref{lem:slash} now yields \eqref{eqn:holslash} while \eqref{eqn:intabs} further implies that $\left|\pi_{\kappa}(f)(z)\right| v^r$ is bounded on $\H$.
\end{proof}

The next proposition constitutes the main step used to prove the modularity of $\widehat{\Psi}$ as a function of $z$ claimed in Theorem \ref{thm:Thetamodular}.
\begin{proposition}\label{prop:holproj}
Suppose that $f:\H\to\C$ is a translation invariant function for which $\left|f(z)\right|v^r$ is bounded on $\H$ and $\kappa\in \frac{1}{2}\Z$ satisfies $\kappa>r+1>1$.  If $\pi_{\kappa}(f)=0$ and $L_z(f)$ is modular of weight $\kappa-2$ for $\Gamma\subseteq\SL_2(\Z)$, then $f$ is modular of weight $\kappa$ for $\Gamma$.
\end{proposition}
\begin{proof}
Since $L_z$ commutes with the slash operator, the modularity of $L_z(f)$ implies that for $M\in \Gamma$
\[
L_z\left(f\Big|_{\kappa}M-f\right)=L_z\left(f\right)\Big|_{\kappa-2}M-L_z\left(f\right)= 0.
\]
Thus,
\[
f\Big|_{\kappa}M-f
\]
is holomorphic.  
Hence by Lemma \ref{lem:hol} we have 
$$
\pi_{\kappa}\left(f\Big|_{\kappa}M-f\right)=f\Big|_{\kappa}M-f.
$$
However, combining Lemma \ref{lem:holconverge} with the fact that $\pi_{\kappa}(f)=0$ then yields
\[
0=\pi_{\kappa}\left(f\right)\Big|_{\kappa}M-\pi_{\kappa}\left(f\right)=\pi_{\kappa}\left(f\Big|_{\kappa}M-f\right)=f\Big|_{\kappa}M-f.
\]
\end{proof}

\section{Convergence and singularities of $\widehat{\Psi}$}\label{sec:converge}

We first prove absolute convergence of $\widehat{\Psi}$.  The following lemma proves useful for this purpose as well as providing the growth conditions necessary to apply holomorphic projection in the next section.
\begin{lemma}\label{lem:posdefbnd}
Suppose that $Q^+$ is a positive definite ternary quadratic form and $v>0$.  Then the sum
\begin{equation}\label{eqn:posdefbnd}
v^{\frac{k}{2}+1}\sum_{a,b,c\in \Z} \left|a\tau^2+b\tau+c\right|^{k-1} e^{-2\pi Q^{+}(a,b,c)v}
\end{equation}
converges absolutely and is bounded as a function of $v$.
\end{lemma}
\begin{remark}
This lemma can be proven more generally with an arbitrary number of variables and an arbitrary homogeneous polynomial.  Without a precise reference, we provide a proof for the special case required in this paper.  The general case would follow by the same argument, but we choose to only include this case to clarify the exposition for the reader.
\end{remark}
\begin{proof}
We first note that
$$
\left|a\tau^2+b\tau+c\right|\ll_{\tau} |a|+|b|+|c|.
$$
  Furthermore, since $Q^+$ is positive definite, there exists a constant $\delta>0$ such that
$$
2\pi Q^+\left(a,b,c\right)\geq \delta\left(a^2+b^2+c^2\right).
$$
Therefore, \eqref{eqn:posdefbnd} can be bounded by
$$
v^{\frac{k}{2}+1}\sum_{a,b,c\in \Z} \left(|a|+|b|+|c|\right)^{k-1} e^{-\delta \left(a^2+b^2+c^2\right)v}.
$$
By the binomial theorem, it suffices to bound sums of the type
$$
\sum_{a,b,c\in \N_0} a^{\ell_1}b^{\ell_2}c^{\ell_3} e^{-\delta\left(a^2+b^2+c^2\right)v}
$$
with $\ell_1+\ell_2+\ell_3=k-1$.  Using Proposition 3 of \cite{ZagierMellin}, one obtains for $v\to 0$
\begin{align*}
\sum_{n\in \N} n^{\ell}e^{-\delta n^2v}&=v^{-\frac{\ell}{2}} \sum_{n\in \N} \left(n\sqrt{v}\right)^{\ell}e^{-\delta \left(n\sqrt{v}\right)^2}\\
&\sim v^{-\frac{1}{2}\left(\ell+1\right)} \int_{0}^{\infty} w^{\ell}e^{-\delta w^2}dw\ll v^{-\frac{1}{2}\left(\ell+1\right)}.
\end{align*}
Combining the above bound with the obvious exponential decay of \eqref{eqn:posdefbnd} as $v\to \infty$ then yields the claim of the lemma.
\end{proof}

\begin{proposition}\label{prop:converge}
The sums defining the two summands $\Psi$ and $\Psi^*$ of $\widehat{\Psi}$ in \eqref{eqn:Thetakdef} converge absolutely.
\end{proposition}
\begin{proof}
By (4.6) and (4.11) of \cite{BKW}, one easily deduces that $\mathcal{F}_D$ converges absolutely and grows at most polynomially as a function of $D$.  Therefore $\Psi$ converges absolutely.

We next move to showing the absolute convergence of $\Psi^*$.  By the well known bound
$$
\Gamma\left(\frac{1}{2};r\right)\ll e^{-r} \qquad  (r\geq 0).
$$
one obtains
\begin{equation}\label{eqn:Psi*bnd}
\sum_{a,b,c\in \Z}\left|Q\left(\tau,1\right)\right|^{k-1}\Gamma\left(\frac{1}{2}; 4\pi Q_{\tau}^2v\right)e^{-2\pi Dv}\\
\ll \sum_{a,b,c\in \Z} \left|Q(\tau,1)\right|^{k-1} e^{-2\pi v\left(D+2Q_{\tau}^2\right)}.
\end{equation}
However, \eqref{eqn:Qposdef} is positive definite, and hence Lemma \ref{lem:posdefbnd} implies the absolute convergence of $\Psi^*$.
\end{proof}
We next rewrite $\widehat{\Psi}$ in terms of other special functions.  In order to do so, we fix $\tau_0=x_0+iy_0\in\H$.  For $r\in \R$ we use the Gauss error function
$$
\erf(r):=\frac{2}{\sqrt{\pi}}\int_{0}^r e^{-t^2}dt
$$
to define
\[
g_k(r):=\frac{1}{\Gamma\left(k-\frac{1}{2}\right)}\int_0^\infty \erf\left(rt^{\frac12}\right) e^{-t} t^{k-\frac{3}{2}}dt.
\]

We furthermore formally define
\begin{equation}\label{eqn:Psi1def}
\Psi_1(\tau,z):=-\sum_{\substack{D\in\Z\\ Q\in \QD}} Q(\tau,1)^{k-1}\left(\sgn\left(Q_{\tau_0}\right)-\erf\left(2Q_{\tau}\sqrt{\pi v}\right)\right)e^{2\pi i Dz}
\end{equation}
and
\begin{equation}\label{eqn:Psi2def}
\Psi_2(\tau,z):=\sum\limits_{\substack{D>0\\ Q\in \QD}} Q\left(\tau,1\right)^{k-1} \left(\sgn\left(Q_{\tau_0}\right)-g_k\left(\frac{Q_{\tau}}{\sqrt{D}}\right)\right) e^{2\pi i Dz}.
\end{equation}
We then rewrite $\widehat{\Psi}$ in the following lemma.
\begin{lemma}\label{lem:rewrite}
The sums $\Psi_1$ and $\Psi_2$ are absolutely convergent and
\begin{equation}\label{eqn:rewrite}
\widehat{\Psi} =\Psi_1+\Psi_2.
\end{equation}
\end{lemma}
Before proving Lemma \ref{lem:rewrite}, we first rewrite $g_k$.
\begin{lemma}\label{gkLemma}
We have
\[
g_k(r)=\sgn(r)-\frac{2}{\beta\left(k-\frac{1}{2}, \frac12\right)}\sgn(r)\psi_{k} \left(\frac{1}{1+r^2}\right).
\]
\end{lemma}

\begin{proof}
Using the fact that
\[
\beta(a, b)=\frac{\Gamma(a)\Gamma(b)}{\Gamma(a+b)},
\]
we compute
\begin{equation}\label{eqn:gk'}
g_k'(r)=\frac{2}{\beta\left(k-\frac{1}{2},\frac{1}{2}\right) \left(1+r^2\right)^{k}}.
\end{equation}
Moreover $g_k(0)=0$ since $\erf(0)=0$.  Thus
$$
g_k(r)= \frac{2\sgn(r)}{\beta\left(k-\frac{1}{2},\frac{1}{2}\right)}\int_0^{|r|}\frac{1}{\left(1+t^2\right)^k}dt.
$$
Making the change of variables $t\to \sqrt{\frac{1}{t}-1}$ easily gives the claim of the lemma.
\end{proof}
\begin{proof}[Proof of Lemma \ref{lem:rewrite}]
First recall that
\begin{equation}\label{eqn:erfrewrite}
\erf(\sqrt{\pi}t)=\sgn(t)\left(1-\frac{1}{\sqrt{\pi}}\Gamma\left(\frac{1}{2};\pi t^2\right)\right).
\end{equation}
By Proposition \ref{prop:converge} the function $\widehat{\Psi}$ converges absolutely.  We add
\begin{equation}\label{eqn:sgndiff}
\Psi_3(\tau,z):=\sum_{\substack{D\in \Z\\ Q\in \QD}}\left(\sgn\left(Q_{\tau_0}\right)-\sgn\left(Q_{\tau}\right)\right)Q\left(\tau,1\right)^{k-1}e^{2\pi i Dz}
\end{equation}
to $\Psi$ and subtract it from $\Psi^*$.  We then compare the $D$-th Fourier coefficient (with respect to $e^{2\pi i u}$) on both sides of \eqref{eqn:rewrite}.  Combining \eqref{eqn:erfrewrite} and Lemma \ref{gkLemma}, it remains to show that $\Psi_3$ converges absolutely and that whenever $D\leq 0$
\begin{equation}\label{eqn:sgnD<=0}
\sgn(Q_{\tau})=\sgn(Q_{\tau_0}).
\end{equation}
To show absolute convergence, we rewrite \eqref{eqn:sgndiff} in the notation of Lemma 2.6 of \cite{ZwegersThesis} and then apply Lemma \ref{lem:posdefbnd}.  We set $q(a,b,c)=b^2-4ac$, $c_1=(-1, 2x, -|\tau|^2)$, and $c_2=(-1, 2x_0, -|\tau_0|^2)$.  Denoting the bilinear form associated to $q$ by $B(u_1,u_2)=q(u_1+u_2)-q(u_1)-q(u_2),$ one computes for $w=(a,b,c)$:
\begin{align*}
q\left(c_1\right)&= -4y^2<0,\\
q\left(c_2\right)&=-4y_0^2<0,\\
B\left(c_1,w\right)&= 4yQ_{\tau}, \\
B\left(c_2,w\right)&= 4y_0Q_{\tau_0},\\
B\left(c_1,c_2\right)&= -4\left(\left|\tau_0\right|^2-2xx_0 +\left|\tau\right|^2\right)<0.
\end{align*}
By Lemma 2.6 of \cite{ZwegersThesis}, the quadratic form
\begin{equation}\label{eqn:Q+def}
Q^+(w):=q(w) + \frac{B\left(c_1,c_2\right)}{4q\left(c_1\right)q\left(c_2\right)-B\left(c_1,c_2\right)^2} B\left(c_1,w\right)B\left(c_2,w\right)
\end{equation}
is positive definite.  Moreover, (2.13) of \cite{ZwegersThesis} implies that
\begin{equation}\label{eqn:Psi3conv}
\sum_{a,b,c\in \Z}\left|Q\left(\tau,1\right)\right|^{k-1}\left|\sgn\left(Q_{\tau_0}\right)-\sgn\left(Q_{\tau}\right)\right|e^{-2\pi q(a,b,c)v}\ll \sum_{a,b,c\in \Z} \left|Q(\tau,1)\right|^{k-1}e^{-2\pi Q^+(a,b,c)v}.
\end{equation}
Since $Q^+$ is positive definite, Lemma \ref{lem:posdefbnd} implies that the above sum converges.

To obtain \eqref{eqn:sgnD<=0}, we note that since $Q^+(w)$ is positive definite, for every $w\neq 0$ with $q(w)\leq 0$ we have 
$$
\frac{B\left(c_1,c_2\right)}{4q\left(c_1\right)q\left(c_2\right)-B\left(c_1,c_2\right)^2} B\left(c_1,w\right)B\left(c_2,w\right)> -q(w)\geq 0.
$$
Noting that $B\left(c_1,c_2\right)<0$ and for $\tau\neq\tau_0$
$$
4q\left(c_1\right)q\left(c_2\right)-B\left(c_1,c_2\right)^2=-16\left(\left(y_0^2-y^2\right)^2 + 2\left(x-x_0\right)^2\left(y_0^2+y^2\right)+\left(x-x_0\right)^4\right)<0,
$$
we have 
$$
B\left(c_1,w\right)B\left(c_2,w\right)>0.
$$
Thus we conclude that 
$$
\sgn\left(Q_{\tau}\right)=\sgn\left(B\left(c_1,w\right)\right)=\sgn\left(B\left(c_2,w\right)\right)=\sgn\left(Q_{\tau_0}\right).
$$

\end{proof}

\section{Modularity and holomorphic projection}\label{sec:modularity}
In this section, we prove Theorem \ref{thm:Thetamodular} and Theorem \ref{thm:Thetaops}.  
As indicated before Proposition \ref{prop:holproj}, a key step in determining modularity is to use holomorphic projection.  In order to do so, we first show that $\Psi_1$ satisfies the growth conditions necessary to apply Lemma \ref{lem:holconverge}.  
\begin{lemma}\label{lem:Psi1bound}
The function $v^{\frac{k}{2}+1} \left|\Psi_1(\tau,z)\right|$ is bounded.
\end{lemma}
\begin{proof}
Since $\Psi_1=\Psi^*-\Psi_3$, it suffices to bound $\Psi_3$ and $\Psi^*$.  By \eqref{eqn:Psi3conv} and Lemma \ref{lem:posdefbnd},  $\Psi_3$ may be estimated against a constant times $v^{-\frac{k}{2}-1}$.  We next bound $\Psi^*$ by \eqref{eqn:Psi*bnd}.  Since \eqref{eqn:Qposdef} is positive definite, Lemma \ref{lem:posdefbnd} concludes the proof.
\end{proof}

By Lemma \ref{lem:Psi1bound}, we may now apply holomorphic projection in $z$ to $\Psi_1$.  If the dependence on $\tau$ is clear, then we suppress it in what follows.  We write
\[
\Psi_1(z)=\sum_{D\in\Z} c_{D}(v) e^{2\pi iD z},
\]
where
\[
c_{D}(v):=-\sum_{Q\in \QD}Q(\tau,1)^{k-1}\left(\sgn\left(Q_{\tau_0}\right) -\erf\left(2 Q_{\tau}\sqrt{\pi v}\right)\right).
\]
\begin{lemma}\label{lem:Thetakhp}
One has that
\begin{equation}\label{eqn:rewritehp}
\widehat{\Psi}=\Psi_1-\pi_{k+\frac{1}{2}}\left(\Psi_1\right).
\end{equation}
Thus in particular
$$
\pi_{k+\frac{1}{2}}\left(\widehat{\Psi}\right)=0.
$$
\end{lemma}
\begin{proof}
By \eqref{eqn:rewrite}, the lemma is equivalent to the statement that
$$
\pi_{k+\frac{1}{2}}\left(\Psi_1\right) = -\Psi_2.
$$

By Lemma \ref{lem:Psi1bound} and Lemma \ref{lem:holconverge}, we may apply holomorphic projection to $\Psi_1$ since $k>3$.  Using the definition \eqref{eqn:holprojformal} of holomorphic projection, we compute
\[
\pi_{k+\frac{1}{2}}\left(\Psi_1\right)(z)=\sum_{D\in\N} c_{D} e^{2\pi i Dz}
\]
with
\begin{multline*}
c_{D} =\frac{(4\pi D)^{k-\frac{1}{2}}}{\Gamma\left(k-\frac{1}{2}\right)}\int_0^\infty c_D(v) e^{-4\pi D v} v^{k-\frac{1}{2}}\frac{dv}{v}\\
=-\frac{(4\pi D)^{k-\frac{1}{2}}}{\Gamma\left(k-\frac{1}{2}\right)} \sum_{Q\in \QD}Q(\tau,1)^{k-1} \int_{0}^{\infty}\left(\sgn\left(Q_{\tau_0}\right) -\erf\left(2Q_{\tau}\sqrt{\pi v}\right)\right)e^{-4\pi D v}v^{k-\frac{1}{2}}\frac{dv}{v}.
\end{multline*}
We consider both integrals separately.  The first summand is evaluated immediately by using the integral representation of the gamma function and the result follows by the definition of $g_k$.
\end{proof}

We now prove Theorem \ref{thm:Thetaops}.
\begin{proof}[Proof of Theorem \ref{thm:Thetaops}]
Note that by Lemma \ref{lem:Thetakhp}
$$
L_z\left(\widehat{\Psi}(\tau,z)\right)=L_z\left(\Psi_1(\tau,z)\right),
$$
because $\pi_{k+\frac{1}{2}}(\Psi_1)$ is holomorphic as a function of $z$.  Hence \eqref{eqn:lowerz} follows directly by
$$
\frac{d}{dr}\erf(r) = \frac{2}{\sqrt{\pi}} e^{-r^2}.
$$

In order to prove \eqref{eqn:lowertau1}, we use Lemma \ref{lem:rewrite} and apply $L_{\tau}$ to $\Psi_1$ and $\Psi_2$.  Using the fact that
\begin{equation}\label{eqn:lowertau}
L_{\tau}\left(Q_{\tau}\right)=\frac{1}{2i}Q(\tau,1),
\end{equation}
one obtains
\[
L_{\tau}\left(\Psi_{1}\left(\tau,z\right)\right) = -y^{2k}\Thstr\left(\tau,z\right).
\]

Using \eqref{eqn:gk'} and \eqref{eqn:lowertau}, a short calculation using \eqref{eqn:Qrewrite} shows that
$$
L_{\tau}\left(\Psi_2(\tau,z)\right) =\frac{i y^{2k}}{\beta\left(k-\frac{1}{2},\frac{1}{2}\right)}\Omega(-\overline{\tau}, z).
$$

\end{proof}

We now use the modularity of $\Theta_1$ proven in Lemma \ref{lem:lowermodular} (1) to obtain Theorem \ref{thm:Thetamodular}.
\begin{proof}[Proof of Theorem \ref{thm:Thetamodular}]

By Lemma \ref{lem:lowermodular} (1), we have that 
$$
L_{z}\left(\Psi\right)\in \spc_{k-\frac{3}{2}}.
$$
Furthermore, by Lemma \ref{lem:Thetakhp}, we know that
$$
\pi_{k+\frac{1}{2}}\left(\widehat{\Psi}\right)=0.
$$
Now note that Lemma \ref{lem:holconverge} implies that $\pi_{k+\frac{1}{2}}\left(\Psi_1\right)$ satisfies the same growth conditions as $\Psi_1$.  Hence Lemma \ref{lem:Psi1bound} together with \eqref{eqn:rewritehp} implies that $\widehat{\Psi}$ satisfies the growth conditions necessary to apply Proposition \ref{prop:holproj} and we conclude that $\widehat{\Psi}$ is modular of weight $k+\frac{1}{2}$ in $z$.   Recalling that $k$ is even, a direct inspection of the Fourier expansion yields that Kohnen's plus space condition is satisfied.

The modularity in $\tau$ follows by the same changes of variables given in the proof of Lemma \ref{lem:lowermodular}.  To complete the proof, we note that the function $\widehat{\Psi}$ is real analytic due to the definitions \eqref{eqn:Psi1def} and \eqref{eqn:Psi2def} in the representation \eqref{eqn:rewrite}.

\end{proof}

\end{document}